\documentclass[a4paper,10pt]{article}
\usepackage{mathrsfs}
\usepackage{latexsym}
\usepackage{amsmath,amssymb}
\usepackage{amsthm}
\usepackage{amsfonts}
\usepackage[usenames]{color}
\usepackage{amssymb}
\usepackage{graphicx}
\usepackage{amsmath}
\usepackage{amsfonts}
\usepackage{amsthm}
\usepackage{mathrsfs}
\usepackage{dsfont}
\usepackage{indentfirst}

\newtheoremstyle{mythm}{1.5ex plus 1ex minus .2ex}{1.5ex plus 1ex
minus .2ex}{\kai}{\parindent}{\song\bfseries}{}{1em}{}
\numberwithin{equation}{section}

\newtheorem{theorem}{Theorem}[section]
\newtheorem{lemma}{Lemma}[section]
\newtheorem*{lemma*}{Lemma}

\newtheorem{remark}{Remark} [section]

\allowdisplaybreaks[4]
\begin{document}\title{{\textbf{A priori bounds for a class of semi-linear degenerate elliptic equations}}}
\author{Genggeng Huang\footnote{genggenghuang@fudan.edu.cn}}
\date{}
\maketitle
\begin{center}
(School of Mathematics Sciences,  LMNS,  Fudan University, Shanghai,
China,  200433)
\end{center}
\begin{abstract}
In this paper, we mainly discuss   a priori bounds of the following degenerate elliptic equation, \begin{equation}\label{000}
a^{ij}(x)\partial_{ij}u+b^i(x)\partial_i u +f(x,u)=0,\text{ in
}\Omega\subset\subset R^n,
\end{equation} where $a^{ij}\partial_i \phi\partial_j \phi=0$ on $\partial \Omega$, $\phi$ is the defining function of $\partial \Omega$. Imposing suitable conditions on the coefficients and $f(x,u)$, one can get the $L^\infty$-estimates of \eqref{000} via blow up method.
\end{abstract}
\par Key Words: degenerate elliptic, characteristic, semi-linear elliptic
\par MSC2010: \  \   \   {35J15, 35J61, 35J70}
\section{Introduction}
\setcounter{section}{1} \setcounter{equation}{0}
\setcounter{theorem}{0}\setcounter{lemma}{0}
In the study of the existence and regularity of the elliptic equations, a priori estimates play an important role at least as far as Schauder's work in 1930s. For the uniformly linear elliptic equations, one can get the $L^\infty-$estimates using the barrier function. But for non-linear elliptic equations, this method  fails. In \cite{GS2}, the authors use a scaling argument reminiscent to that used in the theory of Minimal Surfaces to get the a priori estimates for some uniformly semi-linear elliptic equations. This inspires us to consider the degenerate case.
We mainly consider  a priori bounds of the following degenerate elliptic equation which arises from the study of isometric embedding, \begin{equation}\label{101}
a^{ij}(x)\partial_{ij}u+b^i(x)\partial_i u +f(x,u)=0,\text{ in
}\Omega\subset\subset R^n.
\end{equation} Let $\phi\in C^2$ be the defining function of $\partial
\Omega$, namely,
\begin{equation}\label{102}\phi|_{\partial\Omega}=0,
\nabla\phi|_{\partial \Omega}\neq 0,\phi>0\text{ in  } \Omega\cap\mathcal{N}(\partial \Omega)
\end{equation}where $\mathcal{N}(\partial \Omega)$ is a neighborhood of $\partial \Omega$. Also, suppose that \begin{equation}\label{103}
a^{ij}\partial_i\phi\partial_j\phi=0,\nabla(a^{ij}\partial_i\phi\partial_j\phi)\neq 0
\text{ on }\partial \Omega
\in C^{2}
\end{equation} and that for the  eigenvalues of $\{a^{ij}\}$  $\lambda_1$,..., $\lambda_n$, there hold, for some constant $c_0$,
\begin{equation}\label{104} \lambda_1,...,\lambda_{n-1}\geq
c_0>0,\lambda_{n}=m(x)\phi, 0<m(x)\in C(\bar \Omega).
\end{equation}\begin{theorem}\label{thm101}Let \eqref{102}, \eqref{103} and \eqref{104} be
fulfilled.
Suppose that $0<u\in C^2(\Omega)\cap L^{\infty}(\Omega)$ solves
\eqref{101} and that  $a^{ij}\in
C^{2}(\bar{\Omega}),b^i\in C^{1}(\bar{\Omega})$ and $f(x,t)\in C(\bar{\Omega}\times[0,\infty))$
\begin{equation}
\label{604}\underset{t\rightarrow
\infty}{\lim}\frac{f(x,t)}{t^\alpha}=h(x), \text{ uniformly for some
} 1<\alpha<\frac{n+2a+1}{n+2a-3},
\end{equation} where
$0<h(x)\in C(\bar{\Omega})$,\begin{equation}\label{a101}a=\displaystyle \underset{\partial\Omega}{\sup}\frac{b^i\partial_i\phi-\partial_ja^{ij}\partial_{i}
\phi}{\partial_ka^{ij}\partial_i\phi\partial_j\phi\phi^k}+1,
b\triangleq\underset{\partial\Omega}{\inf}\displaystyle\frac{b^i\partial_i\phi-\partial_ja^{ij}\partial_{i}
\phi}{\partial_{k}a^{ij}\partial_i
\phi\partial_j\phi\phi^k}>\frac 12,\phi^k=\frac{\partial_k\phi}{|\nabla \phi|^2}.
\end{equation}
  Then it follows that \begin{equation}\label{888}|u|_{L^\infty}\leq
C\text{ for some constant } C \text{ independent of } u.\end{equation}

\end{theorem}
\begin{remark}\label{remark1}
Define $$g(x)=\frac{b^i\partial_i\phi-\partial_ja^{ij}\partial_{i}
\phi}{\partial_ka^{ij}\partial_i\phi\partial_j\phi\phi^k}\text{ on }\partial \Omega \text{ where } a^{ij}\partial_i\phi\partial_j\phi=0.$$ The invariance of $g(x)$ is proved in \cite{H}. The numerator of $g(x)$ is the well-know Fichera number. The concept of Fichera number is very important when we deal with degenerate elliptic problems with boundary characteristic degenerate. It indicates whether we shall impose boundary condition in such case. This fact was first observed by M.V.Keldy\text{$\check{s}$} in \cite{Kel} and developed by Fichera in \cite{Fic1,Fic2}. The Fichera number also affects the regularities of the solutions up to the boundary, see \cite{H}. For more details of Fichera number, refer to \cite{O}.
\end{remark}
For dimension $n=2$, Theorem \ref{thm101} was proved in \cite{HU}. The main difficulties arise from the degeneracy on the boundary and no boundary condition. In \cite{HU}, we establish regularity results up to the boundary to overcome the difficulties. The method used in \cite{HU} can't improve the regularity for $n\geq 3$. We use De Giorgi-Moser method to get the H\"older regularity on the boundary. It should be emphasized that we didn't get the uniform H\"older regularity in $\bar \Omega$ but on $\partial \Omega$.  \par Also when we consider one of the blow up equation,
\begin{equation}\label{105}
yu_{yy}+au_y+\Delta_x u+u^\alpha=0 \text{ in } R^{n}_+,u\in L^\infty_{loc}(\overline{R^{n}_+}).
\end{equation} During the process of blow up, we know that we can only expect $L^\infty$-regularity of blow up solution $u$ in \eqref{105}. As the Liouville type theorem in \cite{HU} was obtained under the assumption $u\in C^2(\overline{R^n_+})$. Thus we must improve the regularity up to the boundary for \eqref{105}. By carefully deal with the boundary and the energy inequality, we improved the regularity.
Hence, we must have the following two theorem to overcome the difficulties.
 \begin{theorem}\label{thm102}
Let \eqref{102}, \eqref{103} and \eqref{104} be
fulfilled. Moreover, $g(x)>0$ on the boundary $\partial \Omega$ which is defined in Remark \ref{remark1}.
Suppose that $0<u\in C^2(\Omega)\cap L^{\infty}(\Omega)$ solves
\eqref{101} and that  $a^{ij}\in
C^{2}(\bar{\Omega}),b^i\in C^{1}(\bar{\Omega})$ and $f(x,t)\in C(\bar{\Omega}\times[0,\infty))$. Then:\begin{equation}
|u(x)-u(y)|\leq C|x-y|^{\beta},\forall x\in\bar \Omega,y\in
\partial \Omega, \text{ for some }\beta\in (0,1),
\end{equation}where $C,\beta$ depends only on $|u|_{\infty},|a^{ij}|_{C^2},|b^i|_{C^1},|\partial \Omega|_{C^2},|f|_{C^0}$
and $n$.\end{theorem}

\begin{theorem}\label{thm103}
Let $u\geq 0$ solves \eqref{105}  with $a>\frac 32,\alpha\geq 1$. Then $u\in C^2(\overline{R^n_+})$.
\end{theorem}

\par
The present paper organizes as the following. In Section 2, by the De Giorgi-Moser method we get the H\"older regularity on the boundary of \eqref{104}. In Section 3, the utility of energy integral helps us improve the regularity. In Section 4, we prove Theorem \ref{thm101} via the  blow up method.

\section{$C^\alpha$ regularity on the boundary}
\setcounter{section}{2} \setcounter{equation}{0}
\setcounter{theorem}{0}\setcounter{lemma}{0}

This section is devoted to the proof of Theorem \ref{thm102}. The proof is based on the De Giorgi-Moser method. We first derive a weighted-Sobolev inequality for later use. Denote by
$\widetilde W^{1,2}(G)$ where $G\subset R^n_+=\{(x_1,...x_n)\in R^n|x_n>0\}$
the completion of all the functions $u \in C^1(\bar G)$ under the
following norm
$$\Big(\int_G (x_n (\partial_n u)^2+|\nabla_{x'}u|^2+u^2)dx\Big)^{\frac 12}.$$
\begin{lemma}\label{lem201}
(1)For all $u\in \widetilde W^{1,2}(G)$ with $u=0$ on $\partial G\cap
R^n_+$ there is a universal constant  C independent of $G $ such
that
\begin{equation}\label{weight1}
\left(\int_{G}|u|^{\frac{2(n+1)}{n-1}}dx\right)^\frac{n-1}{n+1}\le C
\int_{G}(x_n (\partial_n u)^2+|\nabla_{x'}u|^2)dx.
\end{equation} (2)With
$G_1=\{0<x_n<1,|x'|^2<1\}$ for any $\epsilon >0$ there exists a
constant $C_\epsilon$ such that
\begin{equation}\label{weight2}
\int_{G_1}u^2dx\le C_\epsilon \int_{G_1}(x_n(\partial_n u)^2+|\nabla_{x'}u|^2)dx
\end{equation} for all $u\in C^1(\bar G_1)$ subject to $|\{x\in G_1|
u(x)=0\}|\ge \epsilon$.
\end{lemma}
\begin{proof} The proof for  the present lemma is based on the
raising dimension argument. Define a transformation
$$ T_:\quad \quad G \ni (x',x_n)\longrightarrow (x',2\sqrt{x_n})=(x',y)\in T(G)$$
Let $G\subset R^n_+$ and $u\in \widetilde W^{1,2}(G)$.  Lift  $T(G)$
in $R^{n+1}$ by such a way $\widetilde{T(G)} =\{(x',y,z)\in
R^{n+1}|(x',y)\in T(G), 0<z<y\}$.  Then
\begin{equation}\label{weight3} \int_G |u|^pdx=\frac 12\int_{T(G)}|u\circ
T^{-1}|^p y dx'dy=\frac 12 ||u\circ
T^{-1}||^p_{L^p(\widetilde{T((G))})}\end{equation} and
\begin{equation}\label{weight4}
\int_G(x_n(\partial_n u)^2+|\nabla_{x'}u|^2)dx=\frac 12\int_{T(G)}|\nabla_{x',y}(u\circ T^{-1})|^2 y
dx'dy=\frac 12\|\widetilde \nabla( u\circ
T^{-1})\|^2_{L^2(\widetilde{T(G)})} \end{equation} where $\widetilde
\nabla=(\nabla_{x'},\partial_y,\partial_z)$ is the
gradient in $R^{n+1}$. Now let us deal with the first inequality of Lemma \ref{lem201}. It suffices to prove (\ref{weight1}) is valid for all
$u\in C^1(\bar G)$. Let $u\in C^1(\bar G)$ with $u=0$ at $\partial
G\cap R^n_+$. Set $\tilde u=u$ as $(x',x_n)\in G$ and $\tilde u=0$ as
$(x',x_n)\in R^n_+\backslash G$,
$$P(u\circ T^{-1})(x',y,z)=\begin{cases}
\tilde u\circ T^{-1}(x',y,z)&\textrm { in }\widetilde{T(G)} \\
\tilde u\circ T^{-1}(x',z,y)&\textrm { as }z>y>0.\\
\end{cases}$$Then
we can extend $P(u\circ T^{-1})$ to $R^{n+1}$ by symmetric extension
first with respect to the plan $y=0$ and then to the plan
$z=0$. Obviously, $P(u\circ T^{-1})\in
H^1(R^{n+1})\hookrightarrow L^q(R^{n+1})$ where $q=\frac{2(n+1)}{n-1}$. Therefore by
(\ref{weight3}), (\ref{weight4}) and the standard Sobolev embedding
theorem it turns out
\begin{eqnarray*}
\left(\int_G|u|^{\frac{2(n+1)}{n-1}}dx\right)^{\frac{n-1}{n+1}}&=&\frac 18\left(\int_{R^{n+1}}|P(u\circ
T^{-1})|^{\frac{2(n+1)}{n-1}} dx'dydz\right)^{\frac{n-1}{n+1}}\\
&\le& C \|\widetilde \nabla(P(u\circ T^{-1}))\|^2_{L^2(R^{n+1})}\\
&=&8C\int_{\widetilde{T(G)}}|\widetilde \nabla(P(u\circ
T^{-1}))|^2dx'dydz\\
&=&16C\int_G(x_n(\partial_n u)^2+|\nabla_{x'}u|^2)dx
\end{eqnarray*}for some universal constant $C$ independent of $u$.
\par Now let us deal with (\ref{weight2}) which corresponds to the Poincare inequality. Suppose that $u\in C^1(\bar G_1)$ where $G_1=\{0<x_n<1,|x'|^2<1\}$, subject to
$|\{(x',x_n)\in G_1| u(x',x_n)=0\}|\ge \epsilon>0$. Then
$|\{(x',y,z)\in \widetilde{T(G_1)}| (u\circ
T^{-1})(x',y,z)=0\}|\ge \epsilon/C$ for some unversal
constant $C$. Hence using the well-known Poincare inequality we
can get
\begin{eqnarray*}
\int_{G_1}|u|^2dx &=&\frac 12\int_{\widetilde{T(G_1)}}|u\circ
T^{-1}|^2dx'dydz\\
&\le& C_\epsilon \int_{\widetilde{T(G_1)}}|\widetilde \nabla(u\circ T^{-1})|^2dx'dydz\\
&=&2C_\epsilon\int_{G_1}(x_n (\partial_n u)^2+|\nabla_{x'}u|^2)dx
\end{eqnarray*}for some constant $C_\epsilon$ depending only on
$\epsilon$. This completes the proof for the present lemma.
\end{proof}

Set $G_R(x_0)=\{|x_i-(x_0)_i|<\sqrt R,i=1,...,n-1,0<x_n<R\}$ for $x_0=(x_0',0)$. Sometimes, if no confusion occurs, simply write $G_R$.

Consider
\begin{equation}\label{201}
\begin{cases}
\mathscr{L}(v)=\partial_i(a^{ij}\partial_j v)+b^j \partial_jv=f\in L^\infty\text{ in } G_R\\
v=0\    \text{ on }\partial G_R\cap \{x_n>0\}\\
\end{cases}
\end{equation}
The coefficients satisfy that\begin{equation}\label{202}
(a^{ij})_{n\times n}\geq 0,(a^{ij})_{n-1\times n-1}\geq \lambda I\text{ for some constant }\lambda>0,a^{in}=x_nb^{in}, \text{ with }  b^{nn}>0.
\end{equation}Moreover, we assume \begin{equation}a^{ij}\in C^2(\overline{R^n_+}), b^{in},b^{i}\in C^1(\overline{R^n_+}), b^n\geq 0\label{1999}\end{equation}

\begin{lemma}\label{lem202} If the assumptions in \eqref{202} and \eqref{1999} are satisfied. Then \eqref{201} admits a solution $v\in H^1(G_R)\cap W^{2,l}_{loc}(G_R),\forall l\in[2,\infty)$, moreover, there exists $R_0$, such that for $\forall R\in (0,R_0)$, it is valid that $$|v|_{\infty}\leq C|f|_\infty R,\text{ for some universal constant C}.$$
\end{lemma}

\begin{proof}
First we shall compute the  sign-invariable $D$ which is Fichera number like refer to \cite{H} on $x_n=0$ where $\phi(x)=x_n$: $$D=b^i\partial_i\phi+\frac 12\partial_k a^{ij}\partial_i\phi\partial_j\phi\partial_k\phi|\nabla\phi|^{-2}=b^n+\frac 12 b^{nn}>0$$
From \cite{H}, one can get the existence of a solution $v$ such that $v\in H^1(G_R)$. And by the non-degeneracy in the interior of $G_R$, $v\in W^{2,l}_{loc}(G_R)$ follows immediately.
Set $\bar v=(v-k)^+$. Multiplying both sides of \eqref{201} by $-\bar v$ and integrating them on $G_R\cap\{x_n\geq \epsilon\}$, one can see that  \begin{eqnarray}\label{203}
&&-\int_{G_R\cap\{x_n\geq \epsilon\}} f\bar v\nonumber\\&=&\int_{G_R\cap \{x_{n}=\epsilon
\}}a^{jn}\bar v\partial_j \bar v+\int_{G_R\cap \{x_n\geq \epsilon\}}a^{ij}\partial_i\bar v\partial_j \bar v-\sum_1^{n-1}\int_{G_R\cap \{x_n\geq \epsilon\}}b^j\partial_j\bar v\bar v\nonumber\\&+&\frac 12\int_{G_R\cap \{x_{n}=\epsilon
\}} b^n\partial_n\bar v^2+\frac 12\int_{G_R\cap \{x_n\geq \epsilon\}}\partial_n b^n \bar v^2\nonumber\\&\ge & c_1\int_{G_R\cap \{x_n\geq \epsilon\}}(x_{n}|\partial_{n}\bar v|^2+|\nabla_{x'}\bar v|^2)-c_2\int_{G_R\cap \{x_n\geq \epsilon\}} \bar v^2+\int_{G_R\cap \{x_{n}=\epsilon
\}}a^{jn}\bar v\partial_j \bar v
\end{eqnarray}
In getting the last inequality, we have used $b^n\geq 0$ and \begin{equation}
a^{ij}\xi_i\xi_j\geq c_0(x_n\xi_n^2+|\xi'|^2),\text{ if } R \text{ is small enough}
\end{equation}

Claim: we can choose suitable $\epsilon_k\rightarrow 0$ such that$$\int_{G_R\cap \{x_{n}=\epsilon_k\}}a^{jn}\bar v\partial_j \bar v\rightarrow 0.$$
As $\bar v\in L^2(G_R)$, without loss of generality we may assume that \begin{equation}
\int_{\bar G_R\cap\{x_{n}=0\}}\bar v^2\leq c
\end{equation}Then \begin{eqnarray}\label{204}
&&\bar v(x',x_{n})=\int_0^{x_n} \partial_{n}\bar v(x',s) ds+\bar v(x',0)\nonumber\\\Rightarrow &&\|\bar v(x',x_n)\|_{L^2(G_R\cap\{x_{n}=t\})}\leq c(\|\partial_{n}\bar v\|_{L^2(G_R)}+\|\bar v(x',0)\|_{L^2(\bar{G}_R\cap\{x_{n}=0\})})
\end{eqnarray}By now, we have proved that $\bar v$ is uniformly integrable on each $x_n-$section of $G_R$. If $\forall \epsilon>0$, $\exists c_0>0$ such that\begin{eqnarray}
&&\int_{G_R\cap \{x_{n}=\epsilon\}}\epsilon|\partial_j \bar v|^2\geq c_0\nonumber\\
\Rightarrow &&\int_{G_R}|\partial_j \bar v|^2\geq \int_0^R\frac{c_0}{x_{n}}dx_{n}=\infty
\end{eqnarray}which is a contradiction. This means $\exists \epsilon_k\rightarrow 0$ such that \begin{equation}\label{205}
\int_{G_R\cap \{x_{n}=\epsilon_k\}}\epsilon_k|\partial_j \bar v|^2\rightarrow 0
\end{equation}Combining \eqref{204} with \eqref{205}, one can get\begin{equation}
\left|\int_{G_R\cap \{x_{n}=\epsilon_k\}}a^{jn}\bar v\partial_j \bar v\right|\leq c\epsilon_k^\frac 12\left(\int_{G_R\cap \{x_{n}=\epsilon_k\}}\epsilon_k|\partial_j \bar v|^2dx'\right)^\frac 12\left(\int_{G_R\cap \{x_{n}=\epsilon_k\}}|\bar v|^2dx'\right)^\frac 12\rightarrow 0
\end{equation}Thus, the claim is proved. This allows us to take $\epsilon=0$ such that \eqref{203} remains true with the last term vanished. It follows that \begin{equation}\label{fff}\int_{G_R}x_n|\partial_n \bar v|^2+|\nabla_{x'}\bar v|^2\leq c\int_{G_R}|f|\bar v+|\bar v|^2.\end{equation}

By Lemma \ref{lem201} and inequality \eqref{fff}, we see \begin{eqnarray*}
\left(\int_{G_R}\bar v^{\frac{2(n+1)}{n-1}}\right)^{\frac{n-1}{n+1}}&\leq& c\int_{G_R}(x_{n}|\partial_{n}\bar v|^2+|\nabla_{x'}\bar v|^2)\nonumber\\
&\leq &c\left(\int \bar v^{\frac{2(n+1)}{n-1}}\right)^{\frac{n-1}{n+1}}|A(k)|^{\frac 2{n+1}}+|f|_\infty\left(\int \bar v^{\frac{2(n+1)}{n-1}}\right)^{\frac{n-1}{2(n+1)}}|A(k)|^{\frac{n+3}{2(n+1)}}\nonumber\\&\leq &
c\left(\int \bar v^{\frac{2(n+1)}{n-1}}\right)^{\frac {n-1}{n+1}}|A(k)|^{\frac 2{n+1}}+\frac 12\left(\int \bar v^{\frac{2(n+1)}{n-1}}\right)^{\frac {n-1}{n+1}}+c|f|^2_{\infty}|A(k)|^{\frac{n+3}{n+1}}
\end{eqnarray*}where $A_k=\{x\in G_R|v\geq k\}$. This implies that \begin{equation}
(h-k)^2|A(h)|^{\frac {n-1}{n+1}}\leq c|f|^2_{\infty}|A(k)|^{\frac{n+3}{n+1}}
\end{equation}\begin{equation}
|A(h)|\leq \frac{c^{\frac{n+1}{n-1}}|f|^{\frac{2(n+1)}{n-1}}}{(h-k)^{\frac{2(n+1)}{n-1}}}
|A(k)|^{\frac{n+3}{n-1}}
\end{equation}if we take $k\geq k_0$ such that $c|A(k_0)|^{\frac 2{n+1}}\leq \frac 14$.
By Proposition 5.1 in [\cite{G},P79], it follows that \begin{equation}
|A(k_0+d)|=0,\  d=\sqrt c|f|_\infty|A(k_0)|^{\frac{2}{n+1}}
\end{equation} This results that \begin{equation}
\underset{G_R}{\sup} \  v^+ \leq k_0+d\leq k_0+C|f|_\infty R
\end{equation}
It remains to estimate $k_0$. From $k^2|A(k)|\leq \|v^+\|_{L^2(G_R)}^2$, one can deduce that \begin{equation}
c|A(k)|^{\frac 2{n+1}}\leq c\left(\frac{\|v^+\|_{L^2(G_R)}^2}{k^2}\right)^\frac{2}{n+1}\leq \frac 14\text{ if }k\geq k_0=(4c)^{\frac{n+1}4}\|v^+\|_{L^2(G_R)}
\end{equation}
Hence,\begin{eqnarray*}
\underset{G_R}{\sup}\   v^+&\leq& C\|v^+\|_{L^2(G_R)}+c|f|_{L^\infty}R\\ &\leq &cR^{\frac{n+1}4}\underset{G_R}{\sup}\  v^+ +c|f|_{L^\infty}R
\end{eqnarray*}If we choose $R$ small enough such that $cR^{\frac{n+1}4}\leq \frac 12$, the present lemma follows immediately from the last inequality.
\end{proof}
\begin{lemma}\label{lem203} Let the assumptions in Lemma \ref{lem202} be fulfilled and let $v\in H^1(G_1)\cap W_{loc}^{2,l}(G_1),\forall l\in[2,\infty)$. Moreover,
$v$ satisfies $\mathscr{L}(v)\ge 0$ in $G_1$. Then \begin{equation}
\underset{G_{\theta R}}{\sup}\  |v|\le C_\theta(\frac 1{|G_R|}\int_{G_R} v^2)^{\frac 12}
\end{equation}
\end{lemma}\begin{proof}Let $\varphi(x)\in C_c^\infty(G_R\cup\{x_n=0\})$ be the cutoff function. Then     \begin{eqnarray}
0&\ge& \int_{G_1\cap \{x_n\geq \epsilon\}}-\varphi^2v^+\mathscr{L}(v)\nonumber\\&=&\int_{G_1\cap \{x_n\geq \epsilon\}}\varphi^2 a^{ij}\partial_i v^+\partial_j v^+ + 2\int_{G_1\cap \{x_n\geq \epsilon\}}\varphi\varphi_ia^{ij}v^+\partial_j v^++\int_{G_1\cap \{x_n= \epsilon\}}\varphi^2 a^{in} v^+\partial_i v^+ \nonumber\\&+&\frac 12\int_{G_1\cap \{x_n=\epsilon\}}\varphi^2 b^n(v^+)^2-\sum_{j=1}^{n-1}\int_{G_1\cap \{x_n\geq \epsilon\}} \varphi^2b^jv^+\partial_j v^+ \nonumber\\&+&\frac 12\int_{G_1\cap \{x_n\geq \epsilon\}} \varphi^2\partial_n b^n(v^+)^2+\int_{G_1\cap \{x_n\geq \epsilon\}}\varphi\varphi_n b^n(v^+)^2\label{2012}
\end{eqnarray}
Following the same arguments as in Lemma \ref{lem202}, we can select a suitable $\epsilon_k\rightarrow 0$ such that the first boundary integral term $\int_{G_1\cap \{x_n= \epsilon_k\}}\varphi^2 a^{in} v^+\partial_i v\rightarrow 0 $ as $\epsilon_k\rightarrow 0$. Noting $b^n\geq 0$, one can get \eqref{2012} implies that  $$
\int_{G_1}x_n|\varphi\partial_{x_n}v^+|^2+|\varphi\nabla_{x'} v^+|^2dx\leq c\int_{G_1}(x_{n}|\partial_{n}\varphi|^2+|\nabla_{x'}\varphi|^2+\varphi^2+
|\partial_{n}\varphi|)(v^+)^2dx$$

Combining the above inequality with Lemma \ref{lem201}, we get \begin{equation}\label{206}\left[\int_{G_1}(\varphi v^+)^{\frac{2(n+1)}{n-1}}dx\right]^{\frac {n-1}{n+1}}\leq C_1\int_{G_1}(x_{n}|\partial_{n}\varphi|^2+|\nabla_{x'}\varphi|^2+\varphi^2+
|\partial_{n}\varphi|)(v^+)^2dx\end{equation}
For any $\theta\in(0,1)$, we construct $\gamma(\xi)\in C^\infty(-\infty,1)$ with $\gamma=1$ as $\xi\leq \theta$ and $|\gamma'|\leq c/(1-\theta)$ as $|\xi|\geq \theta,\gamma(1)=0$.
Set \begin{eqnarray*}&&p_k=\left(\frac{n+1}{n-1}\right)^k,\theta_k=\theta+\frac{1-\theta}{2^k},k=0,1,2,...\\
&&\varphi_k=\prod_{i=1}^{n-1}\gamma\left(\frac{2^{k+1}|x_i-(x_0)_i|}{\sqrt {R}}+1-2^{k+1}\theta_k\right)\gamma\left(\frac{2^{k+1}x_n}{R}+1-2^{k+1}\theta_k
\right)
\end{eqnarray*}  From the fact that $\mathscr{L}(v^+)\geq 0\Rightarrow \mathscr{L}((v^+)^{p_k})\geq 0$, substituting $\varphi$ and $v^+$ in \eqref{206} by $\varphi_k$ and $(v^+)^{p_k}$ respectively, we can obtain with $G_k=G_{\theta_kR}$, $\varphi_k|_{G_{k+1}}=1$,
$$\left(\int\varphi^2_{k+1}|v^+|^{2p_{k+1}}\right)^{\frac{n-1}{n+1}}\le \frac
{C4^{k+1}}{(1-\theta)^2R}\int \varphi_k^2|v^+|^{2p_k}$$
Therefor,
\begin{equation}
\|v^+\|_{L^{2p_{k+1}}(G_{k+1})}\le \frac{C}{(1-\theta)^{\sum\frac 1{p_k}}R^{\sum\frac{1}{2p_k}}}\left(\int_{G_R}|v^+|^2\right)^{\frac 12}=C(1-\theta)^{-\frac{n+1}2}\left(\frac 1{R^{\frac{n+1}2}}\int_{G_R}|v^+|^2\right)^{\frac 12}
\end{equation} This proves the present lemma.
\end{proof}
\begin{lemma}\label{lem204} Let the assumptions in Lemma \ref{lem202} be fulfilled and let $v\in W^{2,l}_{loc}(G_R)\cap L^\infty(G_R),\forall l\in[2,\infty)$ solves \eqref{201}. Then for $R$ suitable small, one can have $v\in \widetilde{W}^{1,2}(G_R)$.
\end{lemma}\begin{proof}
Set $\eta_{\epsilon}(x_n)\in C^\infty(R^1_+)$ that\begin{equation}
\eta_{\epsilon}(x_n)=\begin{cases} 0,\ 0<x_n<\epsilon\\ 1, \
x_n>2\epsilon,
\end{cases}
\end{equation}with $|D^j\eta_\epsilon|\leq C_j\epsilon^{-j}$ for $x_n\in (\epsilon,2\epsilon)$. Multiplying \eqref{201} by $\eta_\epsilon v$ and integrating by parts, one can get \begin{eqnarray}\label{213}
2\int \eta_\epsilon a^{ij}\partial_i v\partial_j v=\int(\partial_n\eta_\epsilon\partial_ja^{nj}-\eta_\epsilon\partial_jb^j-
\partial_n\eta_\epsilon b^n)v^2+\int a^{nn}\partial_n^2\eta_\epsilon v^2-2\int \eta_\epsilon fv
\end{eqnarray} We can choose $R$ suitable small such that\begin{equation}
\int \eta_\epsilon a^{ij}\partial_i v\partial_j v\geq c_0\int \eta_\epsilon(|\nabla_{x'} v|^2+x_n|\partial_n v|^2)
\end{equation} Selecting one term on the right side of \eqref{213} to estimate, one can get \begin{eqnarray}
\left|\int a^{nn}\partial_n^2\eta_\epsilon v^2\right|\leq c|v|_\infty^2\epsilon^{-2}\int_\epsilon^{2\epsilon}x_ndx_n\leq C
\end{eqnarray}Hence, it's easy to see that the right side of \eqref{213} is bounded uniformly independent of $\epsilon$. Passing $\epsilon\rightarrow 0$, one can prove the present lemma.
\end{proof}
With the above lemmas, we can proceed to prove Theorem \ref{thm102}.

The proof for Theorem \ref{thm102}:

Fixing $x_0\in \partial \Omega$, without loss of generality, one can assume $$\partial_{x_i}\phi(x_0)=0,i=1,...,n-1,\partial_{x_n} \phi(x_0)=1.$$ Set $s_i=x_i,i=1,...,n-1,s_n=\phi(x_1,..,x_n),\forall x\in B_\delta(x_0)\cap \bar\Omega$ for some small $\delta>0$. In the new coordinates, \eqref{101} transforms to \begin{equation}\label{207}
\partial_i(\tilde a^{ij}\partial_j u)+(\tilde b^i-\partial_j\tilde a^{ij})\partial_i u+f(s,u)=0 \text{ in }R^n_+\cap U(s(x_0))
\end{equation} where \begin{eqnarray}\label{212}&&\tilde a^{ij}=a^{ij},i,j=1,...,n-1,\tilde a^{in}=a^{ij}\partial_{x_j} \phi,i=1,...n-1,\tilde a^{nn}=a^{ij}\partial_{x_i}\phi\partial_{x_j}\phi,\nonumber\\ &&\tilde b^i=b^i,i=1,...n-1,\tilde b^n=b^i\partial_{x_i}\phi+a^{ij}\partial_{x_ix_j}\phi,\bar b^i\triangleq\tilde b^i-\partial_j\tilde a^{ij},\partial_i=\partial_{s_i}\end{eqnarray}
 and $U(s(x_0))$ is some neighborhood of $s(x_0)$, simply denoted by $U$.
 We now consider \eqref{207} in $G_R\subset U$ for $R\in(0,R_0]$ small. From \eqref{103} and \eqref{104}, we can see $(\tilde a^{ij})_{i,j\leq n-1}\geq \lambda I$, $a^{in}(s',0)=0,i=1...,n$.
 Noting the condition $g(x)>0$ on the boundary and $\phi(s)=s_n$, one can get $$0<\frac{\tilde b^i\partial_i\phi-\partial_j\tilde a^{ij}\partial_{i}\phi}{\partial_k\tilde a^{ij}\partial_i\phi\partial_j\phi\phi^k}=\frac{\tilde b^n-\partial_n \tilde a^{nn}}{\partial_n \tilde a^{nn}}.$$ By \eqref{103} and \eqref{212}, it is easy to see $\partial_n \tilde a^{nn}=\partial_n(a^{ij}\partial_{x_i}\phi\partial_{x_j}\phi)>0 \text{ on }s_n=0$ and $\bar b^n=\tilde b^n-\partial_j\tilde a^{nj}>0$ noting $a^{in}(s',0)=0,i=1...,n$. All the assumptions in Lemma \ref{lem202} are fulfilled, one can get that there exists  $v\in H^1(G_R)\cap W^{2,p}_{loc}(G_R)$  satisfies \begin{equation}\begin{cases}
\partial_i(\tilde a^{ij}\partial_j v)+\bar b^i\partial_i v+f(s,u)=0 \text{ in }G_R\\
v=0 \text{ on }\partial G_R\cap \{s_n>0\}\end{cases}
\end{equation} such that $|v|_{\infty}\leq CR$.   Set $\tilde u=u-v$. Then $\tilde u$ satisfies \begin{equation}
\begin{cases}
\mathscr{R}(\tilde u)=\partial_i(\tilde a^{ij}\partial_j \tilde u)+\bar b^i\partial_i \tilde u=0 \text{ in }G_R\\
\tilde u=u \text{ on }\partial G_R\cap \{s_n>0\}
\end{cases}
\end{equation}
Set $M(R)=\sup_{G_R} \tilde u$ and $m(R)=\inf_{G_R}\tilde u$ and $\underset{G_R}{Osc} \  \tilde u=M(R)-m(R)$. There are two cases to distinguish\begin{equation}
\text{Case }1:\ |\{s\in G_R|w=2\frac{\tilde u-m(R)}{M(R)-m(R)}\geq 1\}|\geq \frac 12|G_R|
\end{equation} and \begin{equation}
\text{Case }2:\ |\{s\in G_R|w=2\frac{M(R)-\tilde u}{M(R)-m(R)}\geq 1\}|\geq \frac 12|G_R|
\end{equation}Without loss of generality, we can assume Case 1 is valid. It is easy to see that \begin{equation}
\mathscr{R}(w)=0\text{ in }G_{R} \text{ and } 0\leq w\leq 2
\end{equation}Let $\varphi\in C_c^\infty(G_R\cup\{x_n=0\})$ with $\varphi\equiv 1$ in $G_{\theta_1R}$ and $\varphi\equiv 0$ in $G_R\backslash G_{\theta_2R}$ for $(\frac 12)^{\frac{2}{n+1}}<\theta_1<\theta_2<1$. For any $\epsilon>0$, \begin{eqnarray}\label{208}
&&0=\int_{G_R\cap\{s_n\geq t\}}\varphi^2\left(\frac 1{w+\epsilon}-1\right)^+\mathscr{R}(w)\nonumber\\&=&\int_{G_R\cap\{s_n\geq t\}}\varphi^2\tilde a^{ij}\left[\left(\ln \frac 1{w+\epsilon}\right)^+\right]_i\left[\left(\ln
\frac 1 {w+\epsilon}\right)^+\right]_j+\int_{G_R\cap\{s_n\geq t\}}2\varphi\varphi_i(1-w-\epsilon)^+\tilde a^{ij}\left[\left(\ln
\frac 1 {w+\epsilon}\right)^+\right]_j
\nonumber\\&-&\sum_{i=1}^{n-1}\int_{G_R\cap\{s_n\geq t\}}\varphi^2(1-w-\epsilon)^+\bar
b^i\left[\left(\ln \frac 1{w+\epsilon}\right)^+\right]_i -\int_{G_R\cap\{s_n= t\}}\varphi^2\left(\frac{1}{w+\epsilon}-1\right)^+\tilde{a}^{nj}\partial_j w\nonumber\\ &-&\int_{G_R\cap\{s_n\geq t\}}\varphi^2 \bar b^n\partial_n\left[\left(\ln\frac 1{w+\epsilon}\right)^+-(1-w-\epsilon)^+\right]\nonumber\\&\ge &\frac 12\int_{G_R\cap\{s_n\geq t\}}\varphi^2\tilde a^{ij}\left[\left(\ln \frac 1{w+\epsilon}\right)^+\right]_i\left[\left(\ln
\frac 1 {w+\epsilon}\right)^+\right]_j\nonumber\\&+&\underset{t\rightarrow 0}{\underline{\lim}}\int_{G_{R}\cap\{s_n=t\}}\bar b^n\varphi^2
\left[\left(\ln\frac 1{w+\epsilon}\right)^+-(1-w-\epsilon)^+\right]
\nonumber\\&-&C\int_{G_R\cap\{s_n\geq t\}}\tilde a^{ij}\varphi_i\varphi_j-\frac{C_\delta|G_{R}|}
{(\theta_2-\theta_1)^2R}-
\frac{C\delta}{R}
\int_{G_R\cap\{s_n\geq t\}}\varphi^2\left[\left(\ln\frac{1}{w+\epsilon}\right)^+\right]^2\end{eqnarray}
The boundary term $\underset{t\rightarrow 0}{\underline{\lim}}\int_{G_R\cap\{s_n=t\}} \varphi^2\tilde a^{ij}\left(\frac 1{w+\epsilon}-1\right)^+\partial_j w=0$ follows from  Lemma \ref{lem204}, \begin{equation}
v\in H^1(G_R),u\in \widetilde W^{1,2}(G_R)\Rightarrow w\in \widetilde W^{1,2}(G_R)
\end{equation}This implies that $\int_{G_R}s_n|\partial_j w|^2\leq C_{R}$. If $\underset{t\rightarrow 0}{\underline{\lim}}\int_{G_R\cap\{s_n=t\}} s_n^2|\partial_j w|^2\geq c>0$, one can get \begin{equation}
\int_{G_R}s_n|\partial_j w|^2\geq \int_0^R c/s_n ds_n=\infty
\end{equation}which is a contradiction. Hence, for fixed $\epsilon >0$, \begin{eqnarray*}
\underset{t\rightarrow 0}{\underline{\lim}}\int_{G_R\cap \{s_n=t\}}\varphi^2\tilde a^{ij}\left(\frac 1{w+\epsilon}-1\right)^+\partial_j w\leq C_{R,\epsilon}\underset{t\rightarrow 0}{\underline{\lim}}\left(\int_{G_R\cap\{s_n=t\}}s_n^2|\partial_j w|^2\right)^{\frac 12}=0
\end{eqnarray*}
Therefor, from \eqref{208} and $\bar b^n>0$, \begin{eqnarray}\label{209}
&&\int_{G_{\theta_1R}}s_n\left|\partial_n\left(\ln \frac
1{w+\epsilon}\right)^+\right|^2+\left|\nabla_{s'}\left(\ln \frac 1{w+\epsilon}\right)^+\right|^2\big)\nonumber\\&\le&
C \int_{G_{R}}(s_n\varphi_{s_n}^2+|\nabla_{s'}\varphi|^2\big)+
\frac{C_\delta|G_{R}|}{(\theta_2-\theta_1)^2R}+
\frac{C\delta}{R}
\int_{G_{\theta_2R}}\left[\left(\ln\frac{1}{w+\epsilon}\right)^+\right]^2
\nonumber\\&\le& \frac{C_\delta|G_{R}|}{(\theta_2-\theta_1)^2R}+\frac{C\delta}{R}
\int_{G_{\theta_2R}}\left[\left(\ln\frac{1}{w+\epsilon}\right)^+\right]^2
\end{eqnarray}
for some constant $C$ independent of $\epsilon$.  In view of Case 1
it follows
\begin{eqnarray*}
\left|\{(s,t)\in G_{\theta_1R}|\left(\ln \frac 1{w+\epsilon}\right)^+=0\}\right|&\geq &|\{(s,t)\in
G_R|w+\epsilon\ge 1\}|-|G_R\backslash G_{\theta_1R}|\\
&\ge& (\theta_1^{\frac {n+1}2}-\frac 12)|G_R|\geq \epsilon_0|G_R|
\end{eqnarray*} Hence Lemma \ref{lem201} implies \begin{eqnarray}\label{210}
&&\int_{G_{\theta_1R}}\left[\left(\ln \frac 1{w+\epsilon}\right)^+\right]^2\nonumber \\
&\le& CR\int_{G_{{\theta_1R}}}s_n\left|\partial_{s_n}\left(\ln \frac
1{w+\epsilon}\right)^+\right|^2+\left|\nabla_{s'}\left(\ln \frac 1{w+\epsilon}\right)^+\right|^2\nonumber\\&\le&
\frac{C_\delta|G_{R}|}{(\theta_2-\theta_1)^2}+C\delta\int_{G_{\theta_2R}}\left[
\left(\ln \frac 1{w+\epsilon}\right)^+\right]^2
\end{eqnarray}Set $\delta=\frac 1{2C}$. Then \begin{equation}
\int_{G_{\theta_1R}}\left|\left(\ln \frac 1{w+\epsilon}\right)^+\right|^2 dx\leq \frac{C|G_R|}{(1-\theta_1)^2}
\end{equation} As $\mathscr{R}(w+\epsilon)=0\Rightarrow \mathscr{R}((\ln\frac 1{w+\epsilon})^+)\geq 0$, by Lemma \ref{lem203}, \begin{equation}
\sup_{G_{\theta R}}\left(\ln \frac 1{w+\epsilon}\right)^+\le C_\theta \left(\frac
1{|G_R|}\int_{G_R}\left|\left(\ln \frac 1{w+\epsilon}\right)^+\right|^2\right)^{\frac 12}\le
C_2
\end{equation}for some constant independent of $\epsilon$. Therefore
$$\inf_{G_{\theta R}}(w+\epsilon)\ge e^{-C_2}\quad \quad \forall \epsilon>0$$and by the
definition of $w$ we have
$$m(\theta R)-m(R)\ge \frac 1C(M(R)-m(R)) $$
 Adding $M(\theta R)-m(\theta R) $ to
the above inequality and rearranging again soon yield
$$\underset{G_{\theta R}}{Osc} \  \tilde u\le \big(1-\frac 1{C})\underset{G_R}{Osc} \  \tilde u \text { for all }R\in (0,R_0]$$ for some
constant $C$
independent of $\epsilon$. From a standard argument it turns out
$$\underset{G_\rho}{Osc} \  \tilde u\le C\big(\frac \rho
{R}\big)^{\beta}\underset{G_R}{Osc} \  \tilde u\text { for all }0<\rho\le
R$$ for some constant $\beta<1$ under control. This implies that\begin{equation}
\underset{G_\rho}{Osc}\ u\leq c_1\big(\frac \rho
{R}\big)^{\beta}+c_2R\leq c\rho^{\frac{\beta}{1+\beta}}
\end{equation}if we set $\rho=R^{1+\frac 1\beta}$.  This proves Theorem \ref{thm102}.

\section{Higher regularities to the boundary for a specific equation}
\setcounter{section}{3} \setcounter{equation}{0}
\setcounter{theorem}{0}\setcounter{lemma}{0}
As one can see, we have proved $C^\alpha$ estimates  on the boundary for a general form of the degenerate elliptic equations. But for a special form of the degenerate elliptic equation, we will have more regularities which are also needed for our later analysis. At first, we  shall give some notations and results about the regularity of solutions to some degenerate elliptic equations due to \cite{HH}.\par Define $I_q(v)$ and $I_\beta(v)$ by:\begin{equation}
I_q(v)=\|y\partial_{yy}v\|_{L^q(R^{n}_+)}+\|\Lambda_1^2v\|_{L^q(R^{n}_+)}+
\|y^{\frac{1}{2}}\Lambda_1
v_y\|_{L^q(R^{n}_+)}+\|v_y\|_{L^q(R^{n}_+)}+\|v\|_{L^q(R^{n}_+)},
\end{equation}\begin{equation}
I_\beta(v)=[y\partial_{yy}v]_{\dot{C}^\beta(\overline{R^{n}_+})}+[
\Lambda_1^2v]_{\dot{C}^\beta(\overline{R^{n}_+})}
+[y^{\frac{1}{2}}\Lambda_1
v_y]_{\dot{C}^\beta(\overline{R^{n}_+})}+[v_y]_{\dot{C}^\beta(\overline{R^{n}_+})}+
\|v\|_{L^\infty(R^{n}_+)},
\end{equation}where $\Lambda_1$ is a singular integral operator with the symbol
$\sigma(\Lambda_1)=|\xi|$. Also we say a function $v(x, y)$ in $\dot{C}^\alpha( \overline{
R^{n}_+})$, $\alpha\in R_+^1\backslash
Z$, if
\begin{equation}|v|_{\dot{C}^\alpha(\overline{R^{n}_+})}=
\sum_{|\beta|\le
[\alpha]}|D^{\beta}v|_{C(\overline{R^{n}_+})}+
[v]_{\dot{C}^{\alpha}(\overline{R^{n}_+})}<\infty,
\end{equation}
where \begin{equation}[v]_{\dot{C}^\alpha(\overline{R^{n}_+})}
=\sum_{|\beta|=[\alpha]}\underset{
y\ge 0, x\neq\bar{x}\in R^{n-1}}{sup}\left(\frac{| D_x^\beta v(x, y)-D_x^\beta
v(\bar{x}, y)|}{|x-\bar{x}|^{\alpha}}\right).\end{equation}
Set
\begin{eqnarray}
\hat K(y, s, \xi)&=&\hat K_1(y,s,\xi)+\hat K_2(y,s,\xi)\nonumber\\\hat K_1(y,s,\xi)&=&
\varphi_1(|\xi|^2y)\varphi_2(|\xi|^2s)(|\xi|^2s)^{a-1}\chi_{(y, +\infty)}(s)\nonumber\\
\hat
K_2(y,s,\xi)&=&\varphi_2(|\xi|^2y)\varphi_1(|\xi|^2s)(|\xi|^2s)^{a-1}\chi_{(0,
y)}(s)\label{c2}\nonumber\\
 \varphi_1(t)&=&t^{\frac{1-a}{2}}I_{a-1}(2\sqrt{t}), \varphi_2(t)=t^{\frac{1-a}{2}}K_{a-1}(2\sqrt{t}),\end{eqnarray} where  $I_{a-1}(t),K_{a-1}(t)$ satisfy the equation $$t^2w_{tt}+tw_t-((a-1)^2+t^2)w=0$$
 Define an operator $K$ as the following:\begin{equation}\label{d009}K(f)(x,y)=\int_0^{+\infty}K(y, s, \cdot)*f(s, \cdot)\mathrm{d}s\end{equation}
 Suppose $f\in L^p(R^n_+)$ has compact support in $\overline{R^n_+}$ and $a>\frac 32$, we will have \begin{equation}\label{k01}
 K(f)\rightarrow 0 \text{ as }(x,y)\rightarrow \infty,\text{ moreover, }I_{p}(K(f))\leq C\|f\|_{L^p},p\geq 2
\end{equation}The proof of \eqref{k01} is due to \cite{HH}. See Theorem 1.2, Lemma 5.2 and Lemma 5.3 in \cite{HH}. Also we list some other lemmas in \cite{HH} for later use.

Let $\psi\in C_c^{\infty}(\overline{R^{n}_+})$ be
 a cutoff function with $\psi(x,y)=1 $ as $|x|\le 1/2 $, $y\le 1/2 $ and $\psi=0$ as $|x|\ge 1$
 or $|y|\ge 1$. Set $\psi_r(x,y)=\psi(\frac xr,\frac yr)$.

\begin{lemma}\label{lem602}(Lemma 5.4 in \cite{HH}) Suppose that $u\in
C^2(R_+^{n})\bigcap L^p(R_+^{n})$ with $u_x,
yu_y\in L^p(R_+^{n})$ satisfies
\begin{equation}\begin{split}\label{611}
L(u)=yu_{yy}+\sum_{i,
j}a_{ij}u_{x_ix_j}+y\sum_ja_ju_{yx_j}+\sum_jb_ju_{x_j}+bu_y=f,
\text{in} \   R_+^{n}, \end{split}
\end{equation}where $a_{ij}, a_j, b_j, b$ are all in
$C(\overline{R_+^{n}})$ with $a_{ij}(0)=\delta_{ij},
b(0)>\frac 32$, $f\in L^\infty(R^{n}_+)$ and that for some
$\epsilon>0$,
\begin{equation}\label{612}
\lim_{y\rightarrow 0}y^{b(0)-1-\epsilon}u(x, y)=0\text { uniformly
for all }x\in R^{n-1}.
\end{equation}
Then for sufficiently large $p$, there are $r=r(p)>0$ such that
\begin{equation} I_p(\psi_ru)\leq C_r,
\end{equation}for some constant depending only on
$p, \|\psi_{2r}f\|_{L^p}, \|\psi_{2r}u\|_{L^p},
\|\psi_{2r}u_x\|_{L^p}$ and $\|y\psi_{2r}u_y\|_{L^p}$ provided that
$p>n$ or $p>\frac {n}2 $ and $b(0)-2-\epsilon>0$.
\end{lemma}
\begin{lemma}\label{lem603}(Lemma 5.5 in \cite{HH})
Suppose that $w, \partial_xw,y\partial_yw \in\dot
{C}^{\alpha}_{loc}(\overline{R_+^{n}})\cap C^2(R_+^{n})$ with
$\alpha\in R_{+}^1 \backslash Z$ and $w$ satisfies
(\ref{611}),where $a_{ij}, a_j, b_j, b, f$ are all in
$\dot{C}^\alpha_{loc}(\overline{R_+^{n}})$ with
$a_{ij}(0)=\delta_{ij}, b(0)>\frac 32$.  Then
\begin{equation}\label{613}
I_{\alpha}(\psi_rw)\le C,
\end{equation}
for some positive constants $r$ and $C$,  depending on  $\alpha,
|\psi_{2r}f|_{\alpha}, |\psi_{2r}w |_{\alpha}$,  $
|\psi_{2r}\partial_xw |_{\alpha}$ and
$|y\psi_{2r}\partial_yw|_{\alpha}$.
\end{lemma}

Denote by $ W^{1,p}_{\alpha}(U)$ the completion of the space of all
the functions $u$ in $C^1(\bar U)$  under the norm
$$\Big(\int_{U}y^{p\alpha}|Du|^pdxdy+\int_{U}y^{p\alpha}|u|^pdxdy\Big)^{\frac 1{p}}.$$
Here we always assume $U\subset R_+^{n}$,  bounded and $\partial
U\cap \{y=0\}$ nonempty.
\begin{lemma}\label{lem604}(Lemma 8.3 in \cite{HH} Appendix B)
Let $U\in C^1$ be bounded domain and let $\alpha\in (0, 1)$,
$p>1/(1-\alpha)$. Then the following maps are continuous
\begin{eqnarray}
W^{1, p}_{\alpha}(U)&\hookrightarrow &C^{\beta}(\bar U)\text { where
}\beta=1-\alpha-\frac {n}p,  \text { if }
p>\frac {n}{(1-\alpha)},\label{A61}\\
W^{1, p}_{\alpha}(U)&\hookrightarrow &L^q(U)\text { where }q<\frac
{np}{n-(1-\alpha)p}, \text { if } p<\frac
{n}{(1-\alpha)}.\label{A71}
\end{eqnarray}
\end{lemma}
Denote by $\overline{W}^{2,p}(\overline{R^n_+})$ the completion of the space of all the functions $u\in C_c^\infty(\overline{R^n_+})$ under the norm $I_p(u)$.

\begin{lemma}\label{lem301}
If $u\in L^\infty_{loc}(\overline{R^n_+})$ is a nonnegative solution of \eqref{105} with $a>\frac 32,\alpha\geq 1$, we must have $u\in \overline W^{2,2}_{loc}(\overline{R^n_+})$.
\end{lemma}
\begin{proof}
By the similar arguments as we did in Lemma \ref{lem204}, one can see that  $u\in \widetilde W^{1,2}_{loc}(\overline{R^n_+})$. Let $\psi(x,y)\in C_c^\infty(\overline{R^n_+})$ be a cutoff function and $\psi_r(x,y)=\psi(\frac x{\sqrt r},\frac yr)$. Set $u_r=\psi_r u$. Then $u_r$ satisfies that \begin{eqnarray}\label{301}
&&y\partial_{yy}u_r+a\partial_y u_r+\Delta_x u_r\\&=&2y\partial_y\psi_r\partial_yu+a\partial_y\psi_r u+2\nabla_x\psi_r\nabla_x u+y\partial_{yy}\psi_r u+\Delta_x \psi_r u-\psi_r u^\alpha=f
\end{eqnarray}Set $v=K(f)$ where $K(f)$ is defined in \eqref{d009}. As $f\in L^2(R^n_+)$ with compact support, it is valid that  $I_2(v)\leq C$ by \eqref{k01}. It follows that for $w=u_r-v$, there holds \begin{equation}
y\partial^2_y w+a\partial_y w+\Delta_x w=0 \text{ in } R^n_+
\end{equation}Multiplying the above equation by $-\psi_R^2w$ and integrating by parts, one can see\begin{eqnarray}
\label{302}&&\int_{R_x^{n-1}\times\{y\geq \epsilon\}}\left(y\psi_R^2(\partial_y w)^2+\psi_R^2|\nabla_x w|^2\right)+\frac{a-1}{2}\int_{R_x^{n-1}\times \{y=\epsilon\}} \psi_R^2w^2\nonumber\\=&&-\int_{R_x^{n-1}\times\{y=\epsilon\}}\epsilon\psi_R^2
w\partial_yw-2\int_{R_x^{n-1}\times\{y\geq\epsilon\}}y\psi_Rw\partial_y\psi_R   \partial_y w\nonumber\\&&-(a-1)\int_{R_x^{n-1}\times\{y\geq\epsilon\}}\partial_y\psi_R\psi_R
w^2-2\int_{R_x^{n-1}\times\{y\geq\epsilon\}}\nabla_x\psi_R\cdot\nabla_x w\psi_R w\nonumber\\ \leq &&-\int_{R_x^{n-1}\times\{y=\epsilon\}}\epsilon\psi_R^2
w\partial_yw+\frac 12\int_{R_x^{n-1}\times\{y\geq\epsilon\}}y\psi_R^2(\partial_y w)^2\nonumber\\&&+\frac 12\int_{R_x^{n-1}\times\{y\geq\epsilon\}}\psi_R^2|\nabla_x w|^2+\frac CR\int_{R^n_+} w^2
\end{eqnarray}
The remaining thing is to prove \begin{equation}\label{303}\underset{\epsilon\rightarrow 0}{\underline{\lim}}\int_{R_x^{n-1}\times\{y=\epsilon\}}\epsilon\psi_R^2
w\partial_yw=0\end{equation} Without loss of generality, it can be assumed that $\|v(x,0)\|_{L^2(R^{n-1}_x)}\leq c$. Hence, \begin{eqnarray}
\label{304} && v(x,y)=\int_0^y\partial_y v dt +v(x,0)\nonumber\\
\Rightarrow &&\|v(x,\epsilon)\|_{L^2(R^{n-1}_x)}\leq c\sqrt\epsilon\|v_y\|_{L^2(R^n_+)}+2\|v(x,0)\|_{L^2(R^{n-1}_x)}
\end{eqnarray} Noting $|u_r|_\infty\leq c$, one can get \begin{equation}\label{305}\int_{R_x^{n-1}\times\{y=\epsilon\}}\psi_R^2
w^2\leq c\text{ independent of }\epsilon\end{equation} Also, as $y^\frac 12\partial_y w \in L^2(R_x^{n-1}\times(0,1))$, it follows from the same arguments in Lemma \ref{201} that \begin{equation}\label{306}
\underset{\epsilon\rightarrow 0}{\underline{\lim}}\int_{R_x^{n-1}\times\{y=\epsilon\}}\epsilon^2\psi_R^2
|\partial_yw|^2=0
\end{equation}Combining \eqref{305} with \eqref{306}, one can see \eqref{303} is valid. Choosing a suitable subsequence of $\epsilon_k\rightarrow 0$ in \eqref{302} yields that \begin{equation}\label{307}
\int_{R^{n}}y\psi_R^2(\partial_y w)^2+\psi_R^2|\nabla_x w|^2 \leq \frac CR\int_{R^n_+} w^2
\end{equation}Passing $R\rightarrow \infty$ in \eqref{307}, one can get $D_{x,y}w\equiv 0$ which means $w\equiv $const. From the compact support of $u_r$ and the asymptotic property of $v$, $w$ must vanish in $R^n_+$. This completes the proof of the present lemma.
\end{proof}

With the aid of Lemma \ref{lem301}, we can now prove Theorem \ref{thm103}.

The proof for Theorem \ref{thm103}:
by Lemma \ref{lem301}, Lemma \ref{lem201} and Sobolev embedding theorem, one can get \begin{equation}
y\partial_y u_r\in H^1(R_+^n),\nabla_x u_r\in \widetilde W^{1,2}(R^n_+)\Rightarrow y\partial_y u_r,\nabla_x u_r\in L^\theta(R^n_+),\theta=\frac{2(n+1)}{n-1}
\end{equation}Set $v=K(f)$. It follows from \eqref{k01} that $v\in \overline W^{2,2}(R^n_+)\cap \overline W^{2,\theta}(R^n_+)$. By the same arguments as in Lemma \ref{lem301}, we can get $u_r=v\in \overline W^{2,2}(R^n_+)\cap \overline W^{2,\theta}(R^n_+)$ with $\nabla_x u_r\in W_{\frac 12}^{1,\theta}(R^n_+)$. Again, by Lemma \ref{lem604} and Sobolev embedding theorem, it yields that $$y\partial_y u_r,\nabla_x u_r\in L^{\frac{n\theta}{n-\frac \theta2}}(R^n_+),\partial_y u_r\in L^{\theta}(R^n_+)
$$Set $\theta_k=\frac{n\theta_{k-1}}{n-\frac{\theta_{k-1}}{2}}, \theta_0=\theta$. Repeating the above arguments, one can get $$
y\partial_y u_r,\nabla_x u_r\in L^{\theta_k}(R^n_+),\partial_y u_r\in L^{\theta_{k-1}}(R^n_+),k=1,2...$$ Choosing $k$ so large such that $ \theta_k>2n$, we must have $$v\in \overline W^{2,2}(R^n_+)\cap \overline W^{2,\theta_k}(R^n_+),y\partial_y u_r,\nabla_x u_r,u_r\in C^\beta(\overline{R^n_+}),\gamma=\frac 12-\frac{n}{\theta_k}$$
Now we can apply Lemma \ref{lem603} to $u_r$, we can get $I_{\gamma}(u_r)\leq C$. This means $$y\partial_y^2 u_r,\Lambda_1^2 u_r,y^{\frac 12}\partial_y\Lambda_1 u_r,\partial_y u_r\in \dot{C}^{\gamma}(\overline{R^n_+}).$$ Applying Lemma \ref{lem603} to $\Lambda_1 u_r$ yields that $I_{\gamma}(\Lambda_1 u_r)\leq C$, or$$
y\partial^2_y \Lambda_1 u_r,\Lambda_1^3 u_r,\partial_y\Lambda_1 u_r,y^{\frac 12}\partial_y\Lambda_1^2 u_r\in \dot{C}^\gamma(\overline{R^n_+}).
$$Applying Lemma \ref{lem603} to $\partial_y u_r$ again, one can get$$
y\partial^3_y u_r,\Lambda^2_1\partial_y u_r,\partial_y^2 u_r,y^{\frac 12}\partial_y^2\Lambda_1 u_r\in \dot{C}^\gamma(\overline{R^n_+}).
$$

This implies that $u_r\in C^2(\overline{R^n_+})$.   Now we have finished the proof of Theorem \ref{thm103}.\begin{remark}As we already know $u\in C^2(\overline{R^n_+})$, by the maximum principle in \cite{HU}, we have $u>0\in C^2(\overline{R^n_+})$.
Proceeding the above arguments and by induction, we can get $u\in C^\infty(\overline{R^n_+})$.
\end{remark}

\section{The proof of  a priori bounds}
\setcounter{section}{4} \setcounter{equation}{0}
\setcounter{theorem}{0}\setcounter{lemma}{0}
Before proving the main theorem, we introduce two Liouville type lemmas concerning the semi-linear elliptic equations. \begin{lemma}\label{lem401}
Let $u\geq 0$ be a $C^2$ solution of \begin{equation*}
\Delta u +u^\alpha=0 \text{ in } R^n.
\end{equation*}Then $u\equiv 0$, for $1\leq\alpha<\frac{n+2}{n-2}$.
\end{lemma}For the proof of Lemma \ref{lem401}, we refer to \cite{GS} and \cite{CGS}. Another Lemma concerns the degenerate case. \begin{lemma}
\label{lem402}Let $u(x,y)\geq 0$ be a $C^2(\overline{R^n_+})$
solution of
\begin{equation*}  yu_{yy}+au_y+\Delta u+u^\alpha=0\text{ in }R^n_+\end{equation*} with $a>1$. Then $u\equiv 0$ for $1\leq \alpha<\frac{n+2a+1}{n+2a-3}$.
\end{lemma}This lemma was proved in author's another paper \cite{HU}.

With all the above preparations, we are now in a position to prove Theorem \ref{thm101}.
We shall use the blow up method to prove Theorem \ref{thm101}. If Theorem \ref{thm101} is not valid, we may assume $|u^k|_\infty=M_k\rightarrow \infty $ as $k\rightarrow \infty$. Hence one can choose  $P_k\in \Omega$ such that  $u(P_k)\geq
\frac{M_k}{2}$. There are two cases we shall distinguish with.

Case 1. $d(P_k,\partial\Omega)\geq \delta>0$. This is the non-degenerate  case and it is easy
to deal with. In fact the final blow-up function satisfies \begin{equation*}
\Delta u+ u^\alpha=0 \text{ in } R^n, u(0)\geq \frac 12
\end{equation*}By Lemma \ref{lem401}, we get a contradiction.   For details, we refer to \cite{GS} and \cite{HU}.

Case 2. $d(P_k,\partial\Omega)\rightarrow 0 $ as $k\rightarrow \infty$. $\underset{k\rightarrow\infty}{\lim} P_k=P\in
\partial\Omega$. Without loss of generality, we may assume $\partial_i \phi(P)=0,i=1,...,n-1
,\partial_{n}\phi\neq 0$ and $P$ to be the origin. Let $s_i=x_i,i=1,..,n-1,s_{n}=\phi(x_1,...,x_{n})$ in $B_\delta(0)\cap \Omega$. Equation \eqref{101}
can be transformed to \begin{equation}
\tilde{a}^{ij}\partial_{ij} u^k+\tilde b^i\partial_i u^k+f(
s,u^k)=0\text{ in } B_{\delta'}(0)\cap \{s_n>0\}.
\end{equation} The coefficients are defined the same as in \eqref{212}.   Set $$\lambda_k^{\frac 1{\alpha-1}}M_k=1,p_i=\frac{s_i-s_i^{P_k}}{\sqrt{\lambda_k}},
i=1,...n-1,p_{n}=\frac{s_{n}-s^{P_k}_{n}}{\lambda_k},v^k(p)=\lambda_k^{\frac 1{\alpha-1}}u^k(s)$$ Then $v^k$ should satisfy, in $H_k=B_{\frac{\delta'}{2\sqrt{\lambda_k}}}(0)\cap \{p_n>-\frac{s_n^{P_k}}{\lambda_k}\}$\begin{eqnarray}\label{413}
\lambda_k^{-1}\tilde a^{nn}\partial_{n}^2 v^k&+&2\sum_{i=1}^n\lambda_k^{-\frac 12}\tilde a^{in}\partial_{in} v^k+\sum_{i,j=1}^{n-1}\tilde a^{ij} \partial_{ij} v^k\nonumber\\&+&\sum_{i=1}^{n-1} \lambda_k^{\frac 12}\tilde b^i \partial_i v^k+\tilde b^{n}\partial_{n} v^k+\lambda_k^{\frac{\alpha}{\alpha-1}}f(p,\lambda_k^{-\frac 1{\alpha-1}}v^k)=0
\end{eqnarray}
Then we will have the
following lemma.\begin{lemma}
\label{lem403}In the region $H_k$ considered, one has \begin{eqnarray}&&\label{411}(\tilde{a}^{ij})\geq c_0I,\text{ for }1\leq i,j\leq n-1,c_0>0\\ \label{412}&&\tilde{a}^{in}=A^{in}(p)\lambda_k(p_n+\frac{s_n^{P_k}}{\lambda_k})\text{ for }i=1,...,n\end{eqnarray}where \begin{equation}\label{a111}A^{in}(p)\in C^1(\bar{H}_k),A^{nn}(p',-\frac{s_n^{P_k}}{\lambda_k})>0,
\frac{\tilde{b}^n(p',-\frac{s^{P_k}_n}{\lambda_k})}
{A^{nn}(p',-\frac{s^{P_k}_n}{\lambda_k})}>\frac 32,p'=(p_1,...,p_{n-1}).\end{equation}
\end{lemma}\begin{proof}
Noting $\tilde{a}^{nn}=a^{ij}\partial_i\phi\partial_j\phi=0$ on $\{s_n=0\}$, we get \begin{eqnarray}\tilde{a}^{nn}(\bar s)&=&\int_0^1\frac{d(a^{ij}\partial_i\phi\partial_j\phi)(
\bar s',t\bar s_n)}{dt}dt=\bar s_n\int_0^{ 1}\partial_{s_n}(a^{ij}\partial_i\phi\partial_j\phi)(\bar s',t\bar s_n)dt \nonumber\\&=&A^{nn}\lambda_k(\bar p_n+\frac{s^{P_k}_n}{\lambda_k}),\text{ where }A^{nn}=\int_0^{1}\partial_{s_n}(a^{ij}\partial_i\phi\partial_j\phi)(\bar s',t\bar s_n)dt.\end{eqnarray}From \eqref{103}, we see that $\nabla(\tilde a^{nn})\neq 0$ on $\{s_n=0\}$ which implies that $\partial_{s_n}\tilde a^{nn}(s',0)>0$ or $A^{nn}(p',-\frac{s^{P_k}_n}{\lambda_k})>0$ in the region considered. The $C^1$ property of $A^{nn}$ follows from the $C^2$ property of $a^{ij},\phi$ immediately. The last term in \eqref{a111} follows from \eqref{a101}.

\end{proof}

By Lemma \ref{lem403}, \eqref{413} changes to \begin{eqnarray}\label{402}
A^{nn}(p_{n}+\frac{s_{n}^{P_k}}{\lambda_k})\partial_{n}^2 v^k&+&2\sum_{i=1}^{n-1}\lambda_k^{\frac 12} A^{in}(p_{n}+\frac{s_{n}^{P_k}}{\lambda_k})\partial_{in} v^k+\sum_{i,j=1}^{n-1}\tilde a^{ij} \partial_{ij} v^k\nonumber\\&+&\sum_{i=1}^{n-1}\lambda_k^{\frac 12}\tilde b^i \partial_i v^k+\tilde b^{n}\partial_{n} v^k+\lambda_k^{\frac{\alpha}{\alpha-1}}f(p,\lambda_k^{-\frac 1{\alpha-1}}v^k)=0
\end{eqnarray}We must take care of the limit of $\frac{s_{n}^{P_k}}{\lambda_k}$.

\underline{Case 2.1}: $\frac{s_{n}^{P_k}}{\lambda_k}\rightarrow \infty$. Set $q_i=p_i,q_{n}=2\sqrt{p_{n}+s_{n}^{P_k}\lambda_k^{-1}}-2\sqrt{s_{n}^{P_k}
\lambda_k^{-1}}$.¡¡Then with $v^k(0)\geq \frac 12$\begin{eqnarray}\label{401}
\partial_{n}^2 v^k&+&2\sum_{i=1}^{n-1} \bar A^{in}(\lambda_k^{\frac 12}q_{n}+\sqrt {s_{n}^{P_k}})\partial_{in} v^k+\sum_{i,j=1}^{n-1}\bar a^{ij} \partial_{ij} v^k\nonumber\\&+&\sum_{i=1}^{n-1} \lambda_k^{\frac 12}\bar b^i \partial_i v^k+\frac{2\bar b^{n}}{q_{n}+2\sqrt{s_{n}^{P_k}\lambda_k^{-1}}}\partial_{n} v^k+\lambda_k^{\frac{\alpha}{\alpha-1}}\bar f(q,\lambda_k^{-\frac 1{\alpha-1}}v^k)=0,\text{ in }J_k
\end{eqnarray}
where $$\bar A^{in}=\frac{A^{in}}{A^{nn}}, \bar a^{ij}=\frac{\tilde a^{ij}}{A^{nn}}, \bar b^i=\frac{\tilde b^i}{A^{nn}},\bar f=\frac{f}{A^{nn}}$$  It is necessary to show $J_k$ can be chosen arbitrary large as $k\rightarrow \infty$. Since $p\in H_k$, it follows that \begin{eqnarray}
&&|q'|^2+\left[\left(\frac{q_n}{2}+\sqrt{\frac{s_n^{P_k}}{\lambda_k}}\right)^2-
\frac{s_n^{P_k}}{\lambda_k}\right]^2<\frac{\delta'^2}{4\lambda_k}\iff |q'|^2+q_n^2\left(\frac{q_n}{4}+\sqrt{\frac{s_n^{P_k}}{\lambda_k}}\right)^2<
\frac{\delta'^2}{4\lambda_k}
\nonumber\\&&\Rightarrow q_n<\frac{\delta'}{\sqrt{s_n^{P_k}}} \text{ noting that } q_n>-2\sqrt{\frac{s_n^{P_k}}{\lambda_k}}.
\end{eqnarray}From $s_n^{P_k}\rightarrow 0$ and $\frac{s_n^{P_k}}{\lambda_k}\rightarrow \infty$, one can get\begin{equation}
\frac{\delta'^2}{4\lambda_k}=4l_k^2{\max} ^2\left\{\frac{\delta'}{4\sqrt{s_n^{P_k}}},
\sqrt{\frac{s_n^{P_k}}{\lambda_k}}\right\},l_k\rightarrow
\infty\text{ as }k\rightarrow \infty.
\end{equation}Since \begin{eqnarray}
|q'|^2+q_n^2\left(\frac{q_n}{4}+\sqrt{\frac{s_n^{P_k}}{\lambda_k}}\right)^2\leq |q'|^2+4{\max} ^2\left\{\frac{\delta'}{4\sqrt{s_n^{P_k}}},
\sqrt{\frac{s_n^{P_k}}{\lambda_k}}\right\}q_n^2,
\end{eqnarray} we can take $J_k=B_{l_k}(0)\cap\{q_n>-2\sqrt{\frac{s_n^{P_k}}{\lambda_k}}\}$. For any fixed $R>0$, $\exists k_R$, for $k\geq k_R$, one has $B_R(0)\subset J_k$. Hence, \eqref{401} is valid in $B_R(0)$ for $k\geq k_R$. As $\lambda I\leq\{\bar a^{ij}\}\leq \Lambda I$ and $s_n^{P_k},\lambda_k\rightarrow 0$, \eqref{401} is uniformly elliptic in $B_R(0)$. Noting $$|\lambda_k^{\frac 12}\bar b^i|_{C^0(B_R(0))}+\left|\frac{2\bar b^{n}}{q_{n}+2\sqrt{s_{n}^{P_k}\lambda_k^{-1}}}\right|_{C^0(B_R(0))}+|\lambda_k
^{\frac{\alpha}{\alpha-1}}\bar f(q,\lambda^{-\frac 1{\alpha-1}}v^k)|_{\infty}\leq C\text{ independent of } k,$$ by standard elliptic equation theorem, one can get $\|v^k\|_{W^{2,p}(B_R(0))}\leq C\Rightarrow \|v^k\|_{C^{1,\beta}(\overline{B_R(0)})}\leq C$ uniformly, with $\beta=1-\frac np$. Thus, passing $k\rightarrow \infty$, we get
the limit equation \begin{equation}
\Delta v+v^\alpha=0 \text{ in } R^{n}, v(0)\geq \frac 12,\alpha<\frac{n+2a+1}{n+2a-3}<\frac{n+2}{n-2}
\end{equation} By Lemma \ref{lem401}, the above equation only admits trivial solution which is a contradiction to $v(0)\geq \frac 12$

\underline{Case 2.2}: $\frac{s_{n}^{P_k}}{\lambda_k}\rightarrow c_0<\infty$.
Set $q_i=p_i,i=1,...,n-1,q_n=p_n+\frac{s_{n}^{P_k}}{\lambda_k}$. Then \eqref{402} transforms to in $B_{\frac{\delta'}{2\sqrt\lambda_k}}((0',\frac{s_n^{P_k}}{\lambda_k}))\cap R^n_+$ with $v^k(0',\frac{s_n^{P_k}}{\lambda_k})\geq \frac 12$,\begin{eqnarray}\label{403}
A^{nn}q_{n}\partial_{n}^2 v^k&+&2\sum_{i=1}^{n-1}\lambda_k^{\frac 12} A^{in}q_{n}\partial_{in} v^k+\sum_{i,j=1}^{n-1}\tilde a^{ij} \partial_{ij} v^k\nonumber\\&+&\sum_{i=1}^{n-1}\lambda_k^{\frac 12}\tilde b^i \partial_i v^k+\tilde b^{n}\partial_{n} v^k+\lambda_k^{\frac{\alpha}{\alpha-1}}f(q,\lambda_k^{-\frac 1{\alpha-1}}v^k)=0
\end{eqnarray}As $\frac{s_n^{P_k}}{\lambda_k}$ is uniformly bounded, we may assume \eqref{403} is satisfied in $B_{\frac{\delta'}{4\sqrt\lambda_k}}(0)\cap R^n_+$.

\par (a):$c_0>0$, by choosing $\delta$ small enough, \eqref{403} is uniformly elliptic in $B_{\delta}(0',c_0)$, hence, $$\|v^k\|_{W^{2,p}(B_{\delta}(0',c_0))}\leq C\Rightarrow \|v^k\|_{C^{1,\beta}(\overline{B_{\delta}(0',c_0)})}\leq C.$$  \par (b):$c_0=0$, by Theorem \ref{thm102}, one can see  that $$
|v^k(0',\frac{s_n^{P_k}}{\lambda_k})-v^k(0)|\leq C\left(\frac{s_n^{P_k}}{\lambda_k}\right)^{\alpha}$$ Since $\frac{s_n^{P_k}}{\lambda_k}\rightarrow 0$ as $k\rightarrow \infty$, it follows that $v^{k}(0)\geq \frac 14$ for $k$ large enough.
Again by Theorem \ref{thm102}, one can get $v^k(q)\geq \frac 18$, $\forall q\in B_{\tilde\delta}(0)\cap R^n_+$ for some $\tilde\delta$ small.

Therefor, by the above arguments, passing the limit $k\rightarrow \infty$, one can get $$v(0,c_0)=\underset{k\rightarrow \infty}{\lim}v^k(0',\frac{s_n^{P_k}}{\lambda_k})\geq \frac 12,c_0>0,$$ or,$$
v(q)\geq \frac 18,\forall q\in B_{\tilde\delta(0)}\cap R^n_+, c_0=0
$$ and $v$ satisfies
$$A^{nn}(P)q_n\partial^2_n v+\sum_{i,j=1}^{n-1}\tilde a^{ij}(P)\partial_{ij}v+\tilde b^n(P)\partial_n v+h(P)v^\alpha=0\text{ in }R^n_+$$ where $\frac{\tilde b^n(P)}{A^{nn}(P)}>\frac 32$ by \eqref{a101}. By a rotating and a stretching of the coordinates, we have that \begin{equation}
q_n\partial_n^2 v+\Delta_{n-1} v+\bar b\partial_n v +v^\alpha=0\text{ in }R^n_+, a\geq \bar b>\frac 32
\end{equation}Since $\alpha<\frac{n+2a+1}{n+2a-3}\leq \frac{n+2\bar b+1}{n+2\bar b-3}$,by Lemma \ref{lem402}, the above equation admits only trivial solution which is a contradiction. This completes the proof of Theorem \ref{thm101}.

\end{document}